\documentclass[11pt,reqno]{amsart}
\usepackage{hyperref}
\usepackage{url}
\usepackage{amssymb}

\newtheorem{theorem}{Theorem}[section]
\newtheorem{corollary}{Corollary}[section]

\newtheorem{lemma}{Lemma}[section]

\theoremstyle{definition}

\theoremstyle{remark}

\numberwithin{equation}{section}
\begin{document}

\title[Arithmetic Identities and Congruences for Partition Triples  With 3-cores]
 {Arithmetic Identities and Congruences for \\ Partition Triples With 3-cores}

\author{LIUQUAN WANG}
\address{Department of Mathematics, National University of Singapore, Singapore, 119076, SINGAPORE}

\email{wangliuquan@u.nus.edu; mathlqwang@163.com}



\dedicatory{}

\maketitle

\begin{abstract}
Let ${{B}_{3}}(n)$ denote the number of partition triples of $n$ where each partition is 3-core. With the help of generating function manipulations, we find several infinite families of arithmetic identities and congruences for ${{B}_{3}}(n)$. Moreover, let $\omega (n)$ denote the number of representations of a nonnegative integer $n$ in the form $x_{1}^{2}+x_{2}^{2}+x_{3}^{2}+3y_{1}^{2}+3y_{2}^{2}+3y_{3}^{2}$ with ${{x}_{1}},{{x}_{2}},{{x}_{3}},{{y}_{1}},{{y}_{2}},{{y}_{3}}\in \mathbb{Z}.$ We find three arithmetic relations between ${{B}_{3}}(n)$ and $\omega (n)$, such as $\omega (6n+5)=4{{B}_{3}}(6n+4).$
\end{abstract}

\section{Introduction}	
A partition of a positive integer $n$ is any nonincreasing sequence of positive integers whose sum is $n$. For example, $7=4+2+1$ and $\lambda =\{4,2,1\}$ is a partition of 7. A partition $\lambda $ of $n$ is said to be a $t$-core if it has no hook numbers that are multiples of $t$. We denote by ${{a}_{t}}(n)$ the number of partitions of $n$ that are $t$-cores. For convenience, we use the following notation
\[{{(a;q)}_{\infty }}=\prod\limits_{n=1}^{\infty }{(1-a{{q}^{n}})}, \quad \text{\rm{and}}  \quad {{f}_{k}}={{({{q}^{k}};{{q}^{k}})}_{\infty }}.\]
From \cite[Eq.\ (2.1)]{Garvan}, the generating function of ${{a}_{t}}(n)$ is given by
\[\sum\limits_{n=0}^{\infty }{{{a}_{t}}(n){{q}^{n}}}=\frac{f_{t}^{t}}{{{f}_{1}}}.\]

	In particular, for $t=3$, we have
\[\sum\limits_{n=0}^{\infty }{{{a}_{3}}(n){{q}^{n}}}=\frac{f_{3}^{3}}{{{f}_{1}}}.\]

A partition $k$-tuple $({{\lambda }_{1}},{{\lambda }_{2}},\cdots ,{{\lambda }_{k}})$ of $n$ is a $k$-tuple of partitions ${{\lambda }_{1}},{{\lambda }_{2}},\cdots ,{{\lambda }_{k}}$ such that the sum of all the parts equals $n$. For example, let ${{\lambda }_{1}}=\{3,1\},{{\lambda }_{2}}=\{1,1\},{{\lambda }_{3}}=\{1\}$. Then $({{\lambda }_{1}},{{\lambda }_{2}},{{\lambda }_{3}})$ is a partition triple of $7$ since $3+1+1+1+1=7$. A partition $k$-tuple of $n$ with $t$-cores is a partition $k$-tuple  $({{\lambda }_{1}},{{\lambda }_{2}},\cdots ,{{\lambda }_{k}})$ of $n$ where each $\lambda _{i}$ is $t$-core for $i=1,2, \cdots ,k$.

Let ${{A}_{t}}(n)$ (resp. ${{B}_{t}}(n)$) denote the number of bipartitions (resp. partition triples) of $n$ with $t$-cores. Then the generating functions for ${{A}_{t}}(n)$ and ${{B}_{t}}(n)$ are given by
\[\sum\limits_{n=0}^{\infty }{{{A}_{t}}(n){{q}^{n}}}=\frac{f_{t}^{2t}}{f_{1}^{2}}\]
and
\begin{equation}\label{Bt}
\sum\limits_{n=0}^{\infty }{{{B}_{t}}(n){{q}^{n}}}=\frac{f_{t}^{3t}}{f_{1}^{3}}
\end{equation}
respectively.

In 1996, Granville and Ono \cite{Ono} found that
\begin{equation}\label{formula}
{{a}_{3}}(n)={{d}_{1,3}}(3n+1)-{{d}_{2,3}}(3n+1),
\end{equation}
where ${{d}_{r,3}}(n)$ denote the number of divisors of $n$ congruent to $r$ modulo 3. Their proof is based on the theory of modular forms.

Baruah and Berndt \cite{Baruah2} showed that for any nonnegative integer $n$,
\[{{a}_{3}}(4n+1)={{a}_{3}}(n).\]

In 2009,  Hirschhorn and Sellers \cite{Hirschhorn} provided an elementary proof of (\ref{formula}) and as corollaries, they proved some arithmetic identities. For example, let $p\equiv 2$ (mod 3) be prime and let $k$ be a positive even integer. Then, for all $n\ge 0$,
\[{{a}_{3}}\Big({{p}^{k}}n+\frac{{{p}^{k}}-1}{3}\Big)={{a}_{3}}(n).\]

Let $u(n)$ denote the number of representations of a nonnegative integer $n$ in the form ${{x}^{2}}+3{{y}^{2}}$ with $x,y\in \mathbb{Z}$. By using Ramanujan's theta function identities, Baruah and Nath \cite{Nath} proved that
\[u(12n+4)=6{{a}_{3}}(n).\]

In 2014, Lin \cite{Lin} discovered some arithmetic identities about ${{A}_{3}}(n)$. For example, he proved that ${{A}_{3}}(8n+6)=7{{A}_{3}}(2n+1)$. Let $v(n)$ denote the number of representations of a nonnegative integer $n$ in the form $x_{1}^{2}+x_{2}^{2}+3y_{1}^{2}+3y_{2}^{2}$ with ${{x}_{1}},{{x}_{2}},{{y}_{1}},{{y}_{2}}\in \mathbb{Z}$. Lin showed that
\begin{equation}\label{A3}
v(6n+5)=12{{A}_{3}}(2n+1).
\end{equation}

Again, Baruah and Nath \cite{BaruahJNT} generalized (\ref{A3}) and established three infinite families of arithmetic identities involving ${{A}_{3}}(n)$. For example, for any integer $k\ge 1$,
\[{{A}_{3}}\bigg({{2}^{2k+2}}n+\frac{2({{2}^{2k+2}}-1)}{3}\bigg)=\frac{{{2}^{2k+2}}-1}{3}\cdot{{A}_{3}}(4n+2)-\frac{{{2}^{2k+2}}-4}{3}\cdot{{A}_{3}}(n).\]
For more results and details about ${{a}_{3}}(n)$ and ${{A}_{3}}(n)$, see \cite{Baruah1,BaruahJNT,Baruah2,Nath,Hirschhorn,Lin}.

Motivated by their work, we study the arithmetic properties of partition triples with 3-cores. By using some identities of $q$ series, we prove some analogous results. We will show that
\[{{B}_{3}}(4n+1)=3{{B}_{3}}(2n), \quad {{B}_{3}}(3n+2)=9{{B}_{3}}(n), \quad \text{\rm{and}}\]
\[{{B}_{3}}(4n+3)=3{{B}_{3}}(2n+1)+4{{B}_{3}}(n).\]
From these relations we deduce three infinite families of arithmetic identities as well as some Ramanujan-type congruences involving ${{B}_{3}}(n)$. For example, we prove two infinite families of congruences for ${{B}_{3}}(n)$: for $k \ge 1$ and all $n \ge 0$,
\begin{displaymath}
\begin{split}
{{B}_{3}}({{2}^{k+1}}n+{{2}^{k}}-1)&\equiv 0 \pmod{\frac{{{4}^{k+1}}+{{(-1)}^{k}}}{5}},\\
{{B}_{3}}({{3}^{k}}n+{{3}^{k}}-1)&\equiv 0 \pmod{{{3}^{2k}}}, \\
{{B}_{3}}({{3}^{k}}n+2\cdot {{3}^{k-1}}-1)&\equiv 0 \pmod{{3}^{2k-1}}.
\end{split}
\end{displaymath}
We will also prove that
\[\sum\limits_{n=0}^{\infty }{{{B}_{3}}(6n+4){{q}^{n}}}=24\frac{f_{2}^{8}f_{3}^{3}}{f_{1}^{5}},\]
from which we deduce the following two Ramanujan-type congruences:
\[{{B}_{3}}(30n+10)\equiv {{B}_{3}}(30n+28)\equiv 0 \pmod{120}.\]

Furthermore, let $\omega (n)$ denote the number of representations of a nonnegative integer $n$ in the form \[n=x_{1}^{2}+x_{2}^{2}+x_{3}^{2}+3y_{1}^{2}+3y_{2}^{2}+3y_{3}^{2}, \quad {{x}_{1}},{{x}_{2}},{{x}_{3}},{{y}_{1}},{{y}_{2}},{{y}_{3}}\in \mathbb{Z}.\]
We find some interesting arithmetic relations between $\omega (n)$ and $B_{3}(n)$:
\begin{displaymath}
\begin{split}
  & \omega (6n+5)=4{{B}_{3}}(6n+4), \\
 & \omega (12n+2)=12{{B}_{3}}(6n), \\
 & \omega (12n+10)=6{{B}_{3}}(6n+4). \\
\end{split}
\end{displaymath}

In the final section, we introduce a unified notation  $A_{3}^{(k)}(n)$ to denote the number of partition $k$-tuples of $n$ wherein each partition is 3-core. We propose two open questions about that whether we can find some analogous results about $A_{3}^{(k)}(n)$ for all positive integer $k$ or not. This will lead to researches in the future.

\section{Main Results and Proofs}
Setting $t=3$ in (\ref{Bt}), we obtain that
\begin{equation}\label{B3}
\sum\limits_{n=0}^{\infty }{{{B}_{3}}(n){{q}^{n}}}=\frac{f_{3}^{9}}{f_{1}^{3}}.
\end{equation}
The following 2-dissection identities will be important in our proofs.
\begin{lemma}\label{13id}
We have
{\setlength\arraycolsep{2pt}
\begin{eqnarray}
\frac{f_{3}^{3}}{{{f}_{1}}}&=&\frac{f_{4}^{3}f_{6}^{2}}{f_{2}^{2}{{f}_{12}}}+q\frac{f_{12}^{3}}{{{f}_{4}}}, \label{lemid1} \\
\frac{{{f}_{3}}}{f_{1}^{3}}&=&\frac{f_{4}^{6}f_{6}^{3}}{f_{2}^{9}f_{12}^{2}}+3q\frac{f_{4}^{2}{{f}_{6}}f_{12}^{2}}{f_{2}^{7}}, \label{lemid2} \\
\frac{f_{1}^{3}}{{{f}_{3}}}&=&\frac{f_{4}^{3}}{{{f}_{12}}}-3q\frac{f_{2}^{2}f_{12}^{3}}{{{f}_{4}}f_{6}^{2}}, \label{lemid3} \\
\frac{{{f}_{1}}}{f_{3}^{3}}&=&\frac{{{f}_{2}}f_{4}^{2}f_{12}^{2}}{f_{6}^{7}}-q\frac{f_{2}^{3}f_{12}^{6}}{f_{4}^{2}f_{6}^{9}}, \label{lemid4} \\
\frac{1}{f_{1}^{4}}&=&\frac{f_{4}^{14}}{f_{2}^{14}f_{8}^{4}}+4q\frac{f_{4}^{2}f_{8}^{4}}{f_{2}^{10}},\label{lemid05} \\
\frac{1}{f_{1}^{8}}&=&\frac{f_{4}^{28}}{f_{2}^{28}f_{8}^{8}}+8q\frac{f_{4}^{16}}{f_{2}^{24}}+16{{q}^{2}}\frac{f_{4}^{4}f_{8}^{8}}{f_{2}^{20}}. \label{lemid5}
\end{eqnarray}}
\end{lemma}
\begin{proof}
For the proofs of (\ref{lemid1})--(\ref{lemid2}), see \cite[Eq.\ (3.75) and Eq.\ (3.38)]{Xia}.
The proofs of (\ref{lemid3})--(\ref{lemid4}) can be found in \cite{Garvan}.
For a proof of (\ref{lemid05}), see \cite[Eq.\ (2.11)]{Xia}. (\ref{lemid5}) follows by squaring both sides of (\ref{lemid05}).
\end{proof}

\begin{lemma}
We have
{\setlength\arraycolsep{2pt}
\begin{eqnarray}
\frac{f_{2}^{8}f_{3}^{4}}{f_{1}^{4}}&=&\frac{f_{2}^{3}f_{3}^{9}}{f_{1}^{3}{{f}_{6}}}+qf_{6}^{8}, \label{thm1id1} \\
\frac{f_{2}^{5}f_{3}^{4}{{f}_{6}}}{f_{1}^{4}} &=& \frac{f_{3}^{9}}{f_{1}^{3}}+q\frac{f_{6}^{9}}{f_{2}^{3}}. \label{thm1id2}
\end{eqnarray}}
\end{lemma}
\begin{proof}
For convenience, we introduce the following notation
\[ {{[x;q]}_{\infty }}={{(x;q)}_{\infty }}{{(q/x;q)}_{\infty }}, \]
\[{{[{{a}_{1}},\cdots ,{{a}_{n}};q]}_{\infty }}=\prod\limits_{i=1}^{n}{{{[{{a}_{i}};q]}_{\infty }}}. \]
Multiplying both sides of (\ref{thm1id1}) by $f_{1}^{4}{{f}_{6}}$, we know (\ref{thm1id1}) is equivalent to
\begin{equation}\label{mid}
f_{2}^{8}f_{3}^{4}{{f}_{6}}={{f}_{1}}f_{2}^{3}f_{3}^{9}+qf_{1}^{4}f_{6}^{9}.
\end{equation}
Note that
\begin{displaymath}
\begin{split}
  & {{f}_{1}}={{[q,{{q}^{2}};{{q}^{6}}]}_{\infty }}{{({{q}^{3}};{{q}^{6}})}_{\infty }}{{({{q}^{6}};{{q}^{6}})}_{\infty }}, \quad {{f}_{2}}={{[{{q}^{2}};{{q}^{6}}]}_{\infty }}{{({{q}^{6}};{{q}^{6}})}_{\infty }}, \\
 & {{f}_{3}}={{({{q}^{3}};{{q}^{6}})}_{\infty }}{{({{q}^{6}};{{q}^{6}})}_{\infty }}, \quad \quad \quad \quad \quad \quad {{f}_{6}}={{({{q}^{6}};{{q}^{6}})}_{\infty }}. \\
\end{split}
\end{displaymath}
Substituting these expressions into (\ref{mid}) and simplifying, we know (\ref{mid}) is equivalent to
\begin{equation}\label{equi}
[{{q}^{2}};{{q}^{6}}]_{\infty }^{4}={{[q,{{q}^{3}},{{q}^{3}},{{q}^{3}};{{q}^{6}}]}_{\infty }}+q[q;{{q}^{6}}]_{\infty }^{4}.
\end{equation}
From \cite[Exercise 2.16,  p. 61]{Gasper}, we know
\begin{equation}\label{keyid}
{{[x\lambda ,x/\lambda ,\mu v,\mu /v;q]}_{\infty }}={{[xv,x/v,\lambda \mu ,\mu /\lambda ;q]}_{\infty }}+\frac{\mu }{\lambda }{{[x\mu ,x/\mu ,\lambda v,\lambda /v;q]}_{\infty }}.
\end{equation}
Taking $(x,\lambda ,\mu ,v,q)\to ({{q}^{3}},q,{{q}^{2}},1,{{q}^{6}})$ in (\ref{keyid}), we have
\[{{[{{q}^{4}},{{q}^{2}},{{q}^{2}},{{q}^{2}};{{q}^{6}}]}_{\infty }}={{[q,{{q}^{3}},{{q}^{3}},{{q}^{3}};{{q}^{6}}]}_{\infty }}+q{{[{{q}^{5}},q,q,q;{{q}^{6}}]}_{\infty }}.\]
Hence (\ref{equi}) holds and we complete our proof of (\ref{thm1id1}).

From (\ref{lemid1}) and (\ref{lemid2}), we obtain that
\begin{equation}\label{f134}
\frac{f_{3}^{4}}{f_{1}^{4}}=\frac{f_{3}^{3}}{{{f}_{1}}}\cdot \frac{{{f}_{3}}}{f_{1}^{3}}=\frac{f_{4}^{9}f_{6}^{5}}{f_{2}^{11}f_{12}^{3}}+3{{q}^{2}}\frac{{{f}_{4}}{{f}_{6}}f_{12}^{5}}{f_{2}^{7}}+4q\frac{f_{4}^{5}f_{6}^{3}{{f}_{12}}}{f_{2}^{9}}.
\end{equation}
Multiplying both sides by $f_{2}^{8}$, we get
\begin{equation}\label{term1}
\frac{f_{2}^{8}f_{3}^{4}}{f_{1}^{4}}=\frac{f_{4}^{9}f_{6}^{5}}{f_{2}^{3}f_{12}^{3}}+4q\frac{f_{4}^{5}f_{6}^{3}{{f}_{12}}}{{{f}_{2}}}+3{{q}^{2}}{{f}_{2}}{{f}_{4}}{{f}_{6}}f_{12}^{5}.
\end{equation}
Applying (\ref{lemid1}), we obtain that
\begin{equation}\label{term2}
\frac{f_{2}^{3}f_{3}^{9}}{f_{1}^{3}{{f}_{6}}}=\frac{f_{2}^{3}}{{{f}_{6}}}\cdot {{\Big(\frac{f_{3}^{3}}{{{f}_{1}}}\Big)}^{3}}=\frac{f_{4}^{9}f_{6}^{5}}{f_{2}^{3}f_{12}^{3}}+3q\frac{f_{4}^{5}f_{6}^{3}{{f}_{12}}}{{{f}_{2}}}+3{{q}^{2}}{{f}_{2}}{{f}_{4}}f_{6}f_{12}^{5}+{{q}^{3}}\frac{f_{2}^{3}f_{12}^{9}}{f_{4}^{3}{{f}_{6}}}.
\end{equation}
Substituting (\ref{term1}) and (\ref{term2}) into (\ref{thm1id1}), we deduce that
\[\frac{f_{4}^{5}f_{6}^{3}{{f}_{12}}}{{{f}_{2}}}=f_{6}^{8}+{{q}^{2}}\frac{f_{2}^{3}f_{12}^{9}}{f_{4}^{3}{{f}_{6}}}.\]
Replacing ${{q}^{2}}$ by $q$ and then multiplying both sides by $\frac{f_{3}}{f_{1}^{3}}$, we obtain (\ref{thm1id2}).
\end{proof}
\begin{theorem}
For any integer $k\ge 1$, we have
\begin{equation}\label{re1}
{{B}_{3}}\big({{2}^{k+1}}n+{{2}^{k}}-1\big)=\frac{{{2}^{2k+2}}+{{(-1)}^{k}}}{5}\cdot{{B}_{3}}(2n),
\end{equation}
\begin{equation}\label{re2}
{{B}_{3}}\big({{2}^{k+1}}n+{{2}^{k+1}}-1\big)=\frac{{{2}^{2k+2}}+{{(-1)}^{k}}}{5}\cdot{{B}_{3}}(2n+1)+\frac{{{2}^{2k+2}}-4{{(-1)}^{k}}}{5}\cdot{{B}_{3}}(n).
\end{equation}
\end{theorem}

\begin{proof}
Substituting (\ref{lemid1}) into (\ref{B3}), we obtain that
\begin{displaymath}
\begin{split}
   \sum\limits_{n=0}^{\infty }{{{B}_{3}}(n){{q}^{n}}}&={{\Big(\frac{f_{4}^{3}f_{6}^{2}}{f_{2}^{2}{{f}_{12}}}+q\frac{f_{12}^{3}}{{{f}_{4}}}\Big)}^{3}} \\
 & =\Big(\frac{f_{4}^{9}f_{6}^{6}}{f_{2}^{6}f_{12}^{3}}+3{{q}^{2}}\frac{f_{4}f_{6}^{2}f_{12}^{5}}{f_{2}^{2}}\Big)+q\Big(3\frac{f_{4}^{5}f_{6}^{4}f_{12}}{f_{2}^{4}}+{{q}^{2}}\frac{f_{12}^{9}}{f_{4}^{3}}\Big). \\
\end{split}
\end{displaymath}
Extracting the terms involving ${{q}^{2n}}$ and ${{q}^{2n+1}}$, respectively, we get
\begin{equation}\label{B32n}
\sum\limits_{n=0}^{\infty }{{{B}_{3}}(2n){{q}^{n}}}=\frac{f_{2}^{9}f_{3}^{6}}{f_{1}^{6}f_{6}^{3}}+3q\frac{{{f}_{2}}f_{3}^{2}f_{6}^{5}}{f_{1}^{2}},
\end{equation}
and
\begin{equation}\label{B32n1}
\sum\limits_{n=0}^{\infty }{{{B}_{3}}(2n+1){{q}^{n}}}=3\frac{f_{2}^{5}f_{3}^{4}{{f}_{6}}}{f_{1}^{4}}+q\frac{f_{6}^{9}}{f_{2}^{3}}.
\end{equation}
By (\ref{thm1id2}), we deduce that
\begin{displaymath}
\begin{split}
\sum\limits_{n=0}^{\infty }{{{B}_{3}}(2n+1){{q}^{n}}} & =3\frac{f_{3}^{9}}{f_{1}^{3}}+4q\frac{f_{6}^{9}}{f_{2}^{3}} \\
 & =3\sum\limits_{n=0}^{\infty }{{{B}_{3}}(n){{q}^{n}}}+4\sum\limits_{n=0}^{\infty }{{{B}_{3}}(n){{q}^{2n+1}}}. \\
\end{split}
\end{displaymath}
Equating the coefficients of ${{q}^{2n}}$ and ${{q}^{2n+1}}$ on both sides, respectively, we obtain
\begin{equation}\label{rec1}
{{B}_{3}}(4n+1)=3{{B}_{3}}(2n)
\end{equation}
and
\begin{equation}\label{rec2}
{{B}_{3}}(4n+3)=3{{B}_{3}}(2n+1)+4{{B}_{3}}(n).
\end{equation}

We are now able to prove  (\ref{re1})--(\ref{re2}). Note that (\ref{rec1}) and (\ref{rec2}) are (\ref{re1}) and (\ref{re2}), respectively, for $k=1$. Now we prove (\ref{re1}). Replacing $n$ by $2n$ in (\ref{rec2}), we have
\[{{B}_{3}}(8n+3)=3{{B}_{3}}(4n+1)+4{{B}_{3}}(2n).\]
By (\ref{rec1}), this implies
\[{{B}_{3}}(8n+3)=13{{B}_{3}}(2n),\]
which is (\ref{re1}) for $k=2$. Now the proof of (\ref{re1}) follows by mathematical induction.

Next, replacing $n$ by $2n+1$ in (\ref{rec2}), we have
\[{{B}_{3}}(8n+7)=3{{B}_{3}}(4n+3)+4{{B}_{3}}(2n+1).\]
Employing (\ref{rec2}) in the above, we deduce that
\[{{B}_{3}}(8n+7)=13{{B}_{3}}(2n+1)+12{{B}_{3}}(n),\]
which is (\ref{re2}) for $k=2$. Now the proof of (\ref{re2}) can be completed by mathematical induction.
\end{proof}

\begin{corollary}\label{cor1}
For any integer $k\ge 1$, we have
\[{{B}_{3}}\big({{2}^{k+1}}n+{{2}^{k}}-1\big)\equiv 0 \pmod{\frac{{{2}^{2k+2}}+{{(-1)}^{k}}}{5}}.\]
\end{corollary}

Recall that the general Ramanujan's theta function $f(a,b)$ is defined by
\[f(a,b):=\sum\limits_{n=-\infty }^{\infty }{{{a}^{n(n+1)/2}}{{b}^{n(n-1)/2}}},\quad |ab|<1.\]
As some special cases, we have (see \cite{Berndt}, for example)
{\setlength\arraycolsep{2pt}
\begin{eqnarray}
   \varphi (q)&:=&f(q,q)=\sum\limits_{n=-\infty }^{\infty }{{{q}^{{{n}^{2}}}}}=(-q;{{q}^{2}})_{\infty }^{2}{{({{q}^{2}};{{q}^{2}})}_{\infty }}, \label{phi} \\
  \psi (q)&:=&f(q,{{q}^{3}})=\sum\limits_{n=0}^{\infty }{{{q}^{n(n+1)/2}}}=\frac{({{q}^{2}};{{q}^{2}})_{\infty }^{2}}{(q;q)_{\infty }^{2}} \label{psi}
\end{eqnarray}}
and
\begin{equation}\label{pen}
f(-q):=f(-q,-{{q}^{2}})=\sum\limits_{n=-\infty }^{\infty }{{{(-1)}^{n}}{{q}^{n(3n-1)/2}}}={{(q;q)}_{\infty }}.
\end{equation}

\begin{lemma}\label{newlem}
We have
\begin{equation}\label{jacobi}
(q;q)_{\infty }^{3}=P({{q}^{3}})-3q{{({{q}^{9}};{{q}^{9}})}^{3}},
\end{equation}
where
\begin{equation}\label{Pq}
P(q)=\sum\limits_{m=-\infty }^{\infty }{{{(-1)}^{m}}(6m+1){{q}^{m(3m+1)/2}}}=f(-q)\varphi (q)\varphi ({{q}^{3}})+4qf(-q)\psi ({{q}^{2}})\psi ({{q}^{6}}).
\end{equation}
\end{lemma}
\begin{proof}
By Jacobi's identity (see \cite[Theorem 1.3.9]{Berndt}), we have
\[(q;q)_{\infty }^{3}=\sum\limits_{n=0}^{\infty }{{{(-1)}^{n}}(2n+1){{q}^{n(n+1)/2}}}.\]
Note that $\frac{n(n+1)}{2}\equiv 0$ (mod 3) if and only if $n\equiv 0$ (mod 3) or $n\equiv 2$ (mod 3).
And $\frac{n(n+1)}{2}\equiv 1$ (mod 3) if and only if  $n\equiv 1$ (mod 3).
Hence we have the following 3-dissection identity
\[(q;q)_{\infty }^{3}=P({{q}^{3}})+qR({{q}^{3}}).\]
We have
\begin{displaymath}
\begin{split}
   P({{q}^{3}}) &=\sum\limits_{m=0}^{\infty }{{{(-1)}^{3m}}(6m+1){{q}^{3m(3m+1)/2}}}+\sum\limits_{m=0}^{\infty }{{{(-1)}^{3m+2}}(6m+5){{q}^{(3m+2)(3m+3)/2}}} \\
 & =\sum\limits_{m=0}^{\infty }{{{(-1)}^{m}}(6m+1){{q}^{3m(3m+1)/2}}}+\sum\limits_{m=-\infty }^{-1}{{{(-1)}^{m}}(6m+1){{q}^{3m(3m+1)/2}}} \\
 & =\sum\limits_{m=-\infty }^{\infty }{{{(-1)}^{m}}(6m+1){{q}^{3m(3m+1)/2}}}. \\
\end{split}
\end{displaymath}
Replacing ${{q}^{3}}$ by $q$, we obtain
\[P(q)=\sum\limits_{m=-\infty }^{\infty }{{{(-1)}^{m}}(6m+1){{q}^{m(3m+1)/2}}}.\]
From \cite{Pre} we know that
\[P(q)={{(q;q)}_{\infty }}\Bigg(1+6\sum\limits_{n\ge 0}{\Big(\frac{{{q}^{3n+1}}}{1-{{q}^{3n+1}}}-\frac{{{q}^{3n+2}}}{1-{{q}^{3n+2}}}\Big)}\Bigg).\]
From \cite[Theorem 3.7.9]{Berndt} we have
\[1+6\sum\limits_{n\ge 0}{\Big(\frac{{{q}^{3n+1}}}{1-{{q}^{3n+1}}}-\frac{{{q}^{3n+2}}}{1-{{q}^{3n+2}}}\Big)}=\varphi (q)\varphi ({{q}^{3}})+4q\psi ({{q}^{2}})\psi ({{q}^{6}}).\]
Hence (\ref{Pq}) is proved.

Again, we have
\[qR({{q}^{3}})=\sum\limits_{m=0}^{\infty }{{{(-1)}^{3m+1}}(6m+3){{q}^{(3m+1)(3m+2)/2}}}.\]
Dividing both sides by $q$ and replacing ${{q}^{3}}$ by $q$, we deduce that
\[R(q)=-3\sum\limits_{m=0}^{\infty }{{{(-1)}^{m}}(2m+1){{q}^{3m(m+1)/2}}}=-3({{q}^{3}};{{q}^{3}})_{\infty }^{3}.\]
\end{proof}

\begin{lemma}\label{recipro}
We have
\begin{equation}\label{midid}
P{{(q)}^{3}}-27q({{q}^{3}};{{q}^{3}})_{\infty }^{9}=\frac{(q;q)_{\infty }^{12}}{({{q}^{3}};{{q}^{3}})_{\infty }^{3}},
\end{equation}
and
\begin{equation}\label{recipro}
\frac{1}{(q;q)_{\infty }^{3}}=\frac{({{q}^{9}};{{q}^{9}})_{\infty }^{3}}{({{q}^{3}};{{q}^{3}})_{\infty }^{12}}\Big(P{{({{q}^{3}})}^{2}}+3qP({{q}^{3}})({{q}^{9}};{{q}^{9}})_{\infty }^{3}+9{{q}^{2}}({{q}^{9}};{{q}^{9}})_{\infty }^{6}\Big).
\end{equation}
\end{lemma}
\begin{proof}
Let $\omega ={{e}^{2\pi i/3}}$. On the one hand, by Lemma \ref{newlem}, we have
\[(q;q)_{\infty }^{3}(\omega q;\omega q)_{\infty }^{3}({{\omega }^{2}}q;{{\omega }^{2}}q)_{\infty }^{3}=\prod\limits_{k=0}^{2}{\Big(P({{q}^{3}})-3{{\omega }^{k}}qf_{9}^{3}\Big)}=P{{({{q}^{3}})}^{3}}-27{{q}^{3}}f_{9}^{9}.\]
On the other hand, by definition we have
\begin{displaymath}
\begin{split}
  & (q;q)_{\infty }^{3}(\omega q;\omega q)_{\infty }^{3}({{\omega }^{2}}q;{{\omega }^{2}}q)_{\infty }^{3} \\
 & =\prod\limits_{n=1}^{\infty }{{{(1-{{q}^{n}})}^{3}}{{(1-{{\omega }^{n}}{{q}^{n}})}^{3}}{{(1-{{\omega }^{2n}}{{q}^{n}})}^{3}}} \\
 & ={{\Big(\prod\limits_{3|n}{(1-{{q}^{n}})(1-{{\omega }^{n}}{{q}^{n}})(1-{{\omega }^{2n}}{{q}^{n}})}\Big)}^{3}}\cdot \prod\limits_{3 \nmid n}{(1-{{q}^{3n}}}{{)}^{3}} \\
 & =\prod\limits_{n=1}^{\infty }{{{(1-{{q}^{3n}})}^{9}}}\cdot \prod\limits_{n=1}^{\infty }{{{(1-{{q}^{3n}})}^{3}}}/\prod\limits_{n=1}^{\infty }{{{(1-{{q}^{9n}})}^{3}}} \\
 & =\frac{({{q}^{3}};{{q}^{3}})_{\infty }^{12}}{({{q}^{9}};{{q}^{9}})_{\infty }^{3}}. \\
\end{split}
\end{displaymath}
Hence we deduce that
\[P{{({{q}^{3}})}^{3}}-27{{q}^{3}}f_{9}^{9}=\frac{({{q}^{3}};{{q}^{3}})_{\infty }^{12}}{({{q}^{9}};{{q}^{9}})_{\infty }^{3}}.\]
Replacing ${{q}^{3}}$ by $q$ we obtain (\ref{midid}).

By (\ref{jacobi}) we have
\begin{displaymath}
\begin{split}
   \frac{1}{(q;q)_{\infty }^{3}}&=\frac{\prod\limits_{k=1}^{2}{(P({{q}^{3}})-3{{\omega }^{k}}qf_{9}^{3})}}{\prod\limits_{k=0}^{3}{(P({{q}^{3}})-3{{\omega }^{k}}qf_{9}^{3})}} \\
 & =\frac{1}{P{{({{q}^{3}})}^{3}}-27{{q}^{3}}f_{9}^{9}}\cdot \Big(P{{({{q}^{3}})}^{2}}+3qP({{q}^{3}})f_{9}^{3}+9{{q}^{2}}f_{9}^{6}\Big) \\
 & =\frac{f_{9}^{3}}{f_{3}^{12}}\Big(P{{({{q}^{3}})}^{2}}+3qP({{q}^{3}})f_{9}^{3}+9{{q}^{2}}f_{9}^{6}\Big),\\
\end{split}
\end{displaymath}
where the last equality follows from (\ref{midid}).
\end{proof}

\begin{theorem}\label{B3gen}
We have
{\setlength\arraycolsep{2pt}
\begin{eqnarray}
   \sum\limits_{n=0}^{\infty }{{{B}_{3}}(3n){{q}^{n}}}&=& P{{(q)}^{2}}\frac{({{q}^{3}};{{q}^{3}})_{\infty }^{3}}{{{(q;q)}_{\infty }^{3}}}, \label{B3n0}\\
 \sum\limits_{n=0}^{\infty }{{{B}_{3}}(3n+1){{q}^{n}}}& =& 3P(q)\frac{({{q}^{3}};{{q}^{3}})_{\infty }^{6}}{(q;q)_{\infty }^{3}}, \label{B3n1} \\
 \sum\limits_{n=0}^{\infty }{{{B}_{3}}(3n+2){{q}^{n}}}& =& 9\frac{({{q}^{3}};{{q}^{3}})_{\infty }^{9}}{(q;q)_{\infty }^{3}}, \label{B3n2} \\
\end{eqnarray}}
\end{theorem}

\begin{proof}
By (\ref{B3}) and (\ref{recipro}) we have
\[\sum\limits_{n=0}^{\infty }{{{B}_{3}}(n){{q}^{n}}}=\frac{({{q}^{3}};{{q}^{3}})_{\infty }^{9}}{(q;q)_{\infty }^{3}}=\frac{({{q}^{9}};{{q}^{9}})_{\infty }^{3}}{({{q}^{3}};{{q}^{3}})_{\infty }^{3}}\Big(P{{({{q}^{3}})}^{2}}+3qP(q^{3})f_{9}^{3}+9{{q}^{2}}f_{9}^{6}\Big).\]
Extracting the terms involving $q^{3n}$, $q^{3n+1}$ and $q^{3n+2}$, respectively, we get the desired results.
\end{proof}

\begin{theorem}\label{B33prop}
For any integer $k\ge 1$, we have
\[{{B}_{3}}({{3}^{k}}n+{{3}^{k}}-1)={{3}^{2k}}{{B}_{3}}(n).\]
\end{theorem}
\begin{proof}
From (\ref{B3}) and (\ref{B3n2}) we deduce that
\begin{equation}\label{9re}
{{B}_{3}}(3n+2)=9{{B}_{3}}(n).
\end{equation}
This proves the theorem for $k=1$.  Replacing $n$ by $3n+2$ in (\ref{9re}), we deduce that
\[{{B}_{3}}(9n+8)={{3}^{2}}{{B}_{3}}(3n+2)={{3}^{4}}{{B}_{3}}(n),\]
which proves the theorem for $k=2$. The theorem now follows by induction on $k$.
\end{proof}
\begin{corollary}
For any integer $k\ge 1$, we have
\begin{displaymath}
\begin{split}
{{B}_{3}}({{3}^{k}}n+{{3}^{k}}-1)&\equiv 0 \pmod{{{3}^{2k}}},\\
{{B}_{3}}({{3}^{k}}n+2\cdot {{3}^{k-1}}-1)&\equiv 0 \pmod{{3}^{2k-1}}.
\end{split}
\end{displaymath}
\begin{proof}
The first congruence follows immediately from Theorem \ref{B33prop}.

By (\ref{B3n2}), we know that ${{B}_{3}}(3n+1)\equiv 0$ (mod 3), and this proves the second congruence for $k=1$. For $k\ge 2$, by Theorem \ref{B33prop}, we have
\[{{B}_{3}}({{3}^{k}}n+2\cdot {{3}^{k-1}}-1)={{B}_{3}}({{3}^{k-1}}(3n+1)+{{3}^{k-1}}-1)={{3}^{2k-2}}{{B}_{3}}(3n+1)\equiv 0 \pmod{{3}^{2k-1}}.\]
\end{proof}

\end{corollary}

\begin{theorem}\label{B36discection}
We have
{\setlength\arraycolsep{2pt}
\begin{eqnarray}
\sum\limits_{n=0}^{\infty }{{{B}_{3}}(6n){{q}^{n}}}&=&\frac{f_{2}^{10}f_{3}^{9}}{f_{1}^{7}f_{6}^{6}}+16q\frac{f_{2}^{7}f_{6}^{3}}{f_{1}^{4}}+27q\frac{f_{2}^{2}f_{3}^{5}f_{6}^{2}}{f_{1}^{3}} \label{B36n} \\
\sum\limits_{n=0}^{\infty }{{{B}_{3}}(6n+4){{q}^{n}}}&=&24\frac{f_{2}^{8}f_{3}^{3}}{f_{1}^{5}}. \label{B36n4}
\end{eqnarray}}
\end{theorem}
\begin{proof}
From (\ref{phi}) and (\ref{psi}), it is not hard to see that
\begin{equation}\label{phipsi}
\varphi (q)=\frac{f_{2}^{5}}{f_{1}^{2}f_{4}^{2}}, \quad \psi (q)=\frac{f_{2}^{2}}{{{f}_{1}}}.
\end{equation}
By (\ref{Pq}), we have
\begin{equation}\label{Pqexp}
P(q)=\frac{f_{2}^{5}f_{6}^{5}}{{{f}_{1}}f_{3}^{2}f_{4}^{2}f_{12}^{2}}+4q\frac{{{f}_{1}}f_{4}^{2}f_{12}^{2}}{{{f}_{2}}{{f}_{6}}}.
\end{equation}
Substituting (\ref{Pqexp}) into (\ref{B3n0}), we obtain
\begin{equation}\label{B3nexp}
\sum\limits_{n=0}^{\infty }{{{B}_{3}}(3n){{q}^{n}}}=\Big(\frac{f_{2}^{10}f_{6}^{10}}{f_{1}^{5}{{f}_{3}}f_{4}^{4}f_{12}^{4}}+16{{q}^{2}}\frac{f_{3}^{3}f_{4}^{4}f_{12}^{4}}{{{f}_{1}}f_{2}^{2}f_{6}^{2}}\Big)+8q\frac{f_{2}^{4}{{f}_{3}}f_{6}^{4}}{f_{1}^{3}}.
\end{equation}
By (\ref{lemid2}) and (\ref{lemid4}), we have
\begin{equation}\label{f153}
\frac{1}{f_{1}^{5}{{f}_{3}}}={{\Big(\frac{{{f}_{3}}}{f_{1}^{3}}\Big)}^{2}}\cdot \frac{{{f}_{1}}}{f_{3}^{3}}=\Big(\frac{f_{4}^{14}}{f_{2}^{17}{{f}_{6}}f_{12}^{2}}+3{{q}^{2}}\frac{f_{4}^{6}f_{12}^{6}}{f_{6}^{5}f_{2}^{13}}\Big)+q\Big(5\frac{f_{4}^{10}f_{12}^{2}}{f_{2}^{15}f_{6}^{3}}-9{{q}^{2}}\frac{f_{4}^{2}f_{12}^{10}}{f_{6}^{7}f_{2}^{11}}\Big).
\end{equation}
Now substituting (\ref{lemid1}), (\ref{lemid2}) and (\ref{f153}) into (\ref{B3nexp}),  and then extracting the terms involving ${{q}^{2n}}$, we obtain
\[\sum\limits_{n=0}^{\infty }{{{B}_{3}}(6n){{q}^{2n}}}=\frac{f_{6}^{9}f_{4}^{10}}{f_{2}^{7}f_{12}^{6}}+16{{q}^{2}}\frac{f_{4}^{7}f_{12}^{3}}{f_{2}^{4}}+27{{q}^{2}}\frac{f_{4}^{2}f_{6}^{5}f_{12}^{2}}{f_{2}^{3}}.\]
Replacing ${{q}^{2}}$ by $q$ we prove (\ref{B36n}).

Similarly, substituting (\ref{Pqexp}) into (\ref{B3n1}), we obtain
\begin{equation}\label{B3n1exp}
\sum\limits_{n=0}^{\infty }{{{B}_{3}}(3n+1){{q}^{n}}}=3\frac{f_{2}^{5}f_{3}^{4}f_{6}^{5}}{f_{1}^{4}f_{4}^{2}f_{12}^{2}}+12q\frac{f_{3}^{6}f_{4}^{2}f_{12}^{2}}{f_{1}^{2}{{f}_{2}}{{f}_{6}}}.
\end{equation}

By (\ref{lemid1}), we deduce that
\begin{displaymath} \frac{f_{3}^{6}}{f_{1}^{2}}={{\Big(\frac{f_{3}^{3}}{{{f}_{1}}}\Big)}^{2}}=\frac{f_{4}^{6}f_{6}^{4}}{f_{2}^{4}f_{12}^{2}}+2q\frac{f_{4}^{2}f_{6}^{2}f_{12}^{2}}{f_{2}^{2}}+{{q}^{2}}\frac{f_{12}^{6}}{f_{4}^{2}}. \nonumber
\end{displaymath}
Substituting this identity and (\ref{f134}) into (\ref{B3n1exp}), and extracting the terms involving ${{q}^{2n+1}}$, we obtain
\[\sum\limits_{n=0}^{\infty }{{{B}_{3}}(6n+4){{q}^{2n+1}}}=12q\Big(\frac{f_{4}^{3}f_{6}^{8}}{f_{2}^{4}f_{12}}+\frac{f_{4}^{8}f_{6}^{3}}{f_{2}^{5}}+{{q}^{2}}\frac{f_{12}^{8}}{{{f}_{2}}{{f}_{6}}}\Big).\]
Dividing both sides by $q$ and replacing ${{q}^{2}}$ by $q$, and applying (\ref{thm1id1})£¬ we obtain that
\[\sum\limits_{n=0}^{\infty }{{{B}_{3}}(6n+4){{q}^{n}}}=\frac{12}{{{f}_{1}}{{f}_{3}}}\Big(\frac{f_{2}^{3}f_{3}^{9}}{f_{1}^{3}{{f}_{6}}}+qf_{6}^{8}\Big)+12\frac{f_{2}^{8}f_{3}^{3}}{f_{1}^{5}}=24\frac{f_{2}^{8}f_{3}^{3}}{f_{1}^{5}}.\]
\end{proof}
\begin{corollary}
For any integer $n\ge 0$, we have
\[{{B}_{3}}(6n+4)\equiv 0  \pmod{24}.\]
\end{corollary}
\begin{proof}
This follows from (\ref{B36n4}).
\end{proof}
\begin{theorem}
For any integer $n\ge 0$, we have
\[{{B}_{3}}(30n+10)\equiv {{B}_{3}}(30n+28)\equiv 0 \pmod{120}.\]
\end{theorem}
\begin{proof}
Note that for any prime $p\ge 3$, we have
\[\binom{p}{k} = \frac{p}{k}\cdot \binom{p-1}{k-1} \equiv 0 \pmod{p}, \quad 1 \le k \le p-1.\]
By the binomial theorem, we have
\[{{(1-x)}^{p}}=1-px+\cdots +p{{x}^{p-1}}-{{x}^{p}}\equiv 1-{{x}^{p}} \pmod{p}.\]
Hence for any integer $a\ge 1$, we have
\[f_{a}^{p}\equiv \prod\limits_{n=1}^{\infty }{{{(1-{{q}^{an}})}^{p}}}\equiv \prod\limits_{n=1}^{\infty }{(1-{{q}^{apn}})}={{f}_{ap}} \pmod{p}.\]
From (\ref{B36n4}) we have
\begin{equation}\label{temp1}
\sum\limits_{n=0}^{\infty }{{{B}_{3}}(6n+4){{q}^{n}}}=24\frac{f_{2}^{8}f_{3}^{3}}{f_{1}^{5}}\equiv 24\frac{{{f}_{10}}}{{{f}_{5}}}\cdot f_{2}^{3}f_{3}^{3} \pmod{120}.
\end{equation}
By Jacobi's identity, we have
\[f_{2}^{3}f_{3}^{3}=\Big(\sum\limits_{k=0}^{\infty }{{{(-1)}^{k}}(2k+1){{q}^{k(k+1)}}}\Big)\Big(\sum\limits_{l=0}^{\infty }{{{(-1)}^{l}}(2l+1){{q}^{3l(l+1)/2}}}\Big).\]
Suppose
\[f_{2}^{3}f_{3}^{3}=\sum\limits_{m=0}^{\infty }{c(m){{q}^{m}}},\]
then
\[c(m)=\sum\limits_{\begin{smallmatrix}
 k(k+1)+3l(l+1)/2=m \\
 k,l\ge 0
\end{smallmatrix}}{{{(-1)}^{k+l}}(2k+1)(2l+1)}.\]
Note that
\[m=k(k+1)+\frac{3l(l+1)}{2}\Leftrightarrow 8m+5=2{{(2k+1)}^{2}}+3{{(2l+1)}^{2}}.\]
For any integer $x$, we have ${{x}^{2}}\equiv 0,1,4$ (mod 5). If $m\equiv 1$ or 4 (mod 5), then at least one of $2k+1$ or $2l+1$ must be divisible by 5. Hence we deduce that
\[c(5n+1)\equiv c(5n+4)\equiv 0 \pmod{5}.\]
By (\ref{temp1}) we have
\[\sum\limits_{n=0}^{\infty }{{{B}_{3}}(6n+4){{q}^{n}}}\equiv 24\frac{{{f}_{10}}}{{{f}_{5}}}\sum\limits_{m=0}^{\infty }{c(m){{q}^{m}}} \pmod{120}.\]
The theorem now follows by comparing the coefficients of ${{q}^{5n+r}}$ ($r\in \{1,4\}$) on both sides of the above identity.
\end{proof}

\begin{theorem}
Let $\omega (n)$ denote the number of representations of a nonnegative integer $n$ in the form $x_{1}^{2}+x_{2}^{2}+x_{3}^{2}+3y_{1}^{2}+3y_{2}^{2}+3y_{3}^{2}$ with ${{x}_{1}},{{x}_{2}},{{x}_{3}},{{y}_{1}},{{y}_{2}},{{y}_{3}}\in \mathbb{Z}$. Then
\[\omega (6n+5)=4{{B}_{3}}(6n+4).\]
\end{theorem}

\begin{proof}
By \cite[p.\ 49, Corollary (i)]{notebook} and Jacobi triple product identity \cite[Theorem 1.3.3]{Berndt}, we can deduce that
\[\varphi (q)=\varphi ({{q}^{9}})+2q\frac{f_{6}^{2}{{f}_{9}}{{f}_{36}}}{{{f}_{3}}{{f}_{12}}{{f}_{18}}}.\]
The generating function of $\omega (n)$ is given by
\begin{equation}\label{wgen}
\sum\limits_{n=0}^{\infty }{\omega (n){{q}^{n}}}={{\varphi }^{3}}(q){{\varphi }^{3}}({{q}^{3}})={{\varphi }^{3}}({{q}^{3}}){{\Big(\varphi ({{q}^{9}})+2q\frac{f_{6}^{2}{{f}_{9}}{{f}_{36}}}{{{f}_{3}}{{f}_{12}}{{f}_{18}}}\Big)}^{3}}.
\end{equation}
Extracting the terms ${{q}^{3n+2}}$ from (\ref{wgen}), dividing by ${{q}^{2}}$, replacing ${{q}^{3}}$ by $q$, we obtain that
\[\sum\limits_{n=0}^{\infty }{\omega (3n+2){{q}^{n}}}=12{{\varphi }^{3}}(q)\varphi ({{q}^{3}})\frac{f_{2}^{4}f_{3}^{2}f_{12}^{2}}{f_{1}^{2}f_{4}^{2}f_{6}^{2}}=12\frac{f_{2}^{19}f_{6}^{3}}{f_{1}^{8}f_{4}^{8}}.\]
By (\ref{lemid5}) we get
\begin{equation}\label{w3n2}
\sum\limits_{n=0}^{\infty }{\omega (3n+2){{q}^{n}}}=12\frac{f_{2}^{19}f_{6}^{3}}{f_{4}^{8}}\Big(\frac{f_{4}^{28}}{f_{2}^{28}f_{8}^{8}}+8q\frac{f_{4}^{16}}{f_{2}^{24}}+16{{q}^{2}}\frac{f_{4}^{4}f_{8}^{8}}{f_{2}^{20}}\Big).
\end{equation}
If we extract the terms involving ${{q}^{2n+1}}$, divide by $q$ and replace ${{q}^{2}}$ by $q$,  we obtain
\begin{equation}\label{w6n5}
\sum\limits_{n=0}^{\infty }{\omega (6n+5){{q}^{n}}}=96\frac{f_{2}^{8}f_{3}^{3}}{f_{1}^{5}}.
\end{equation}
Comparing (\ref{w6n5}) with (\ref{B36n4}), we complete our proof.
\end{proof}

\begin{theorem}
For any nonnegative integer $n$, we have
\[\omega (12n+2)=12{{B}_{3}}(6n).\]
\end{theorem}
\begin{proof}
If we extract the terms involving ${{q}^{2n}}$ in (\ref{w3n2})  and then replace ${{q}^{2}}$ by $q$, we get
\begin{equation}\label{w6n2}
\sum\limits_{n=0}^{\infty }{\omega (6n+2){{q}^{n}}}=12\Big(\frac{f_{2}^{20}f_{3}^{3}}{f_{1}^{9}f_{4}^{8}}+16q\frac{f_{3}^{3}f_{4}^{8}}{{{f}_{1}}f_{2}^{4}}\Big).
\end{equation}
By (\ref{lemid2}) we obtain
\begin{equation}\label{f3319}
\begin{split}
\frac{f_{3}^{3}}{f_{1}^{9}}&={{\Big(\frac{f_{4}^{6}f_{6}^{3}}{f_{2}^{9}f_{12}^{2}}+3q\frac{f_{4}^{2}{{f}_{6}}f_{12}^{2}}{f_{2}^{7}})}^{3}}\\
&=\Big(\frac{f_{4}^{18}f_{6}^{9}}{f_{2}^{27}f_{12}^{6}}+27{{q}^{2}}\frac{f_{4}^{10}f_{6}^{5}f_{12}^{2}}{f_{2}^{23}}\Big)+9q\Big(\frac{f_{4}^{14}f_{6}^{7}}{f_{2}^{25}f_{12}^{2}}+3{{q}^{2}}\frac{f_{4}^{6}f_{6}^{3}f_{12}^{6}}{f_{2}^{21}}\Big).
\end{split}
\end{equation}
Substituting (\ref{lemid1}) and (\ref{f3319}) into (\ref{w6n2}), extracting the terms involving ${{q}^{2n}}$ and then replacing ${{q}^{2}}$ by $q$, we deduce that
\begin{equation}\label{w12n2}
\sum\limits_{n=0}^{\infty }{\omega (12n+2){{q}^{n}}}=12\Big(\frac{f_{2}^{10}f_{3}^{9}}{f_{1}^{7}f_{6}^{6}}+16q\frac{f_{2}^{7}f_{6}^{3}}{f_{1}^{4}}+27q\frac{f_{2}^{2}f_{3}^{5}f_{6}^{2}}{f_{1}^{3}}\Big).
\end{equation}

Comparing (\ref{B36n}) with (\ref{w12n2}), we deduce that $\omega (12n+2)=12{{B}_{3}}(6n).$
\end{proof}

\begin{theorem}
For any nonnegative integer $n$, we have
\[\omega (12n+10)=6{{B}_{3}}(6n+4).\]
\end{theorem}
\begin{proof}
Extracting the terms involving ${{q}^{3n+1}}$ in (\ref{wgen}), dividing both sides by $q$ and replacing ${{q}^{3}}$ by $q$, we get
\[\sum\limits_{n=0}^{\infty }{\omega (3n+1){{q}^{n}}}=6{{\varphi }^{3}}(q)\varphi {{({{q}^{3}})}^{2}}\cdot \frac{f_{2}^{2}{{f}_{3}}{{f}_{12}}}{{{f}_{1}}{{f}_{4}}{{f}_{6}}}.\]
Substituting (\ref{phipsi}) into the above identity, we get
\begin{equation}\label{w3n1}
\sum\limits_{n=0}^{\infty }{\omega (3n+1){{q}^{n}}}=6\frac{f_{2}^{15}}{f_{1}^{6}f_{4}^{6}}\cdot \frac{f_{6}^{10}}{f_{3}^{4}f_{12}^{4}}\cdot \frac{f_{2}^{2}{{f}_{3}}{{f}_{12}}}{{{f}_{1}}{{f}_{4}}{{f}_{6}}}=6\frac{f_{2}^{17}f_{6}^{9}}{f_{1}^{7}f_{3}^{3}f_{4}^{7}f_{12}^{3}}.
\end{equation}
Substituting (\ref{lemid4}) and (\ref{lemid5}) into (\ref{w3n1}), we obtain
\begin{displaymath}
\begin{split}
  \sum\limits_{n=0}^{\infty }{\omega (3n+1){{q}^{n}}}& =6\frac{f_{2}^{17}f_{6}^{9}}{f_{4}^{7}f_{12}^{3}}\cdot \frac{1}{f_{1}^{8}}\cdot \frac{{{f}_{1}}}{f_{3}^{3}} \\
 & =6\frac{f_{2}^{17}f_{6}^{9}}{f_{4}^{7}f_{12}^{3}}\cdot \Big(\frac{f_{4}^{28}}{f_{2}^{28}f_{8}^{8}}+8q\frac{f_{4}^{16}}{f_{2}^{24}}+16{{q}^{2}}\frac{f_{4}^{4}f_{8}^{8}}{f_{2}^{20}}\Big)\Big(\frac{{{f}_{2}}f_{4}^{2}f_{12}^{2}}{f_{6}^{7}}-q\frac{f_{2}^{3}f_{12}^{6}}{f_{4}^{2}f_{6}^{9}}\Big). \\
\end{split}
\end{displaymath}
Extracting the term involving ${{q}^{2n+1}}$, dividing by $q$ and then replacing ${{q}^{2}}$ by $q$, we get
\begin{equation}\label{w6n4}
\sum\limits_{n=0}^{\infty }{\omega (6n+4){{q}^{n}}}=6\Big(8\frac{f_{3}^{2}f_{2}^{11}}{f_{1}^{6}{{f}_{6}}}-\frac{f_{2}^{19}f_{6}^{3}}{f_{1}^{8}f_{4}^{8}}-16q\frac{f_{4}^{8}f_{6}^{3}}{f_{2}^{5}}\Big).
\end{equation}
By (\ref{lemid2}) we have
\begin{equation}\label{f3216}
\frac{f_{3}^{2}}{f_{1}^{6}}={{\Big(\frac{f_{4}^{6}f_{6}^{3}}{f_{2}^{9}f_{12}^{2}}+3q\frac{f_{4}^{2}{{f}_{6}}f_{12}^{2}}{f_{2}^{7}}\Big)}^{2}}=\frac{f_{4}^{12}f_{6}^{6}}{f_{2}^{18}f_{12}^{4}}+6q\frac{f_{4}^{8}f_{6}^{4}}{f_{2}^{16}}+9{{q}^{2}}\frac{f_{4}^{4}f_{6}^{2}f_{12}^{4}}{f_{2}^{14}}.
\end{equation}
Substituting (\ref{lemid5}) and (\ref{f3216}) into (\ref{w6n4}), extracting the terms involving ${{q}^{2n+1}}$, dividing by $q$ and replacing ${{q}^{2}}$ by $q$, we obtain
\[\sum\limits_{n=0}^{\infty }{\omega (12n+10){{q}^{n}}}=144\frac{f_{2}^{8}f_{3}^{3}}{f_{1}^{5}}.\]
Comparing this identity with (\ref{B36n4}), we deduce that $\omega (12n+10)=6{{B}_{3}}(6n+4).$
\end{proof}

\section{Concluding Remarks}
Let $A_{3}^{(k)}(n)$  denote the number of partition $k$-tuples of $n$
 with 3-cores. In particular, $A_{3}^{(1)}(n)=a_{3}(n)$, $A_{3}^{(2)}(n)=A_{3}(n)$ appeared in the existing literature (see \cite{BaruahJNT,Nath,Hirschhorn,Lin}) and $A_{3}^{(3)}(n)=B_{3}(n)$ in this paper.

 Again, let  ${{\omega }^{(k)}}(n)$ denote the number of representations of a nonnegative integer $n$ in the form
 \[n=x_{1}^{2}+\cdots +x_{k}^{2}+3(y_{1}^{2}+\cdots +y_{k}^{2}), \quad {{x}_{i}},{{y}_{i}}\in \mathbb{Z},\quad i=1,2,\cdots ,k.\]
It is easy to see that the generating function of $A_{3}^{(k)}(n)$ and ${{\omega }^{(k)}}(n)$ are given by
\[\sum\limits_{n=0}^{\infty }{A_{3}^{(k)}(n){{q}^{n}}}=\frac{f_{3}^{3k}}{f_{1}^{k}}, \quad \text{\rm{and}}\quad \sum\limits_{n=0}^{\infty }{{{\omega }^{(k)}}(n){{q}^{n}}}={{\varphi }^{k}}(q){{\varphi }^{k}}({{q}^{3}})\]
respectively.

 From the existing papers and our work, we know many arithmetic identities about $A_{3}^{(k)}(n)$ for $k=1,2,3$. Meanwhile, we have seen some relations between  $A_{3}^{(k)}(n)$ and ${{\omega }^{(k)}}(n)$, such as
\begin{displaymath}
\begin{split}
   {{\omega }^{(1)}}(12n+4)&=6A_{3}^{(1)}(n), \\
  {{\omega }^{(2)}}(6n+5)&=12A_{3}^{(2)}(2n+1), \\
  {{\omega }^{(3)}}(6n+5)&=4A_{3}^{(3)}(6n+4). \\
\end{split}
\end{displaymath}
Based on these facts and observations, we would like to ask the following two questions.

\textbf{Question 1}. Can we find some arithmetic identities involving $A_{3}^{(k)}(n)$ for all $k$?

\textbf{Question 2}. Can we find some arithmetic relations between $A_{3}^{(k)}(n)$ and ${{\omega }^{(k)}}(n)$ for all $k$?

To answer these questions, we believe that one may need to develop some new methods and ideas.

\end{document}